\documentclass{amsproc}
\usepackage{amssymb}
\usepackage{url}
\usepackage{hyperref}
\usepackage{xcolor}

\usepackage{rotating}
\usepackage{pdflscape}

\usepackage{mathrsfs}

\newtheorem{theorem}{Theorem}[section]
\newtheorem{lemma}[theorem]{Lemma}

\newtheorem{proposition}[theorem]{Proposition}
\newtheorem{corollary}[theorem]{Corollary}
\theoremstyle{definition}

\theoremstyle{remark}
\newtheorem{remark}[theorem]{Remark}

\numberwithin{equation}{section}

\newcommand{\F}{\mathbb{F}}
\newcommand{\Z}{\mathbb{Z}}
\newcommand{\Q}{\mathbb{Q}}
\newcommand{\C}{\mathbb{C}}

\newcommand{\SL}{\mathrm{SL}}
\newcommand{\PSL}{\mathrm{PSL}}
\newcommand{\trace}{\mathrm{trace}}
\newcommand{\cl}{\mathrm{cl}}

\begin{document}

\title{Freeness and $S$-arithmeticity 
of rational M\"{o}bius groups}

\author{A.~S.~Detinko}

\author{D.~L.~Flannery}

\author{A.~Hulpke}


\date{}

\begin{abstract}
We initiate a new, computational 
 approach to 
a classical  problem: 
certifying non-freeness of  
($2$-generator, parabolic)
M\"{o}bius subgroups 
of $\SL(2,\Q)$.
The main tools used 
are algorithms for Zariski 
dense groups and algorithms to 
compute a presentation 
of $\SL(2, R)$ for a
localization $R= \Z[\frac{1}{b}]$
of $\Z$.
We prove that a M\"{o}bius 
group $G\leq \SL(2,R)$ is not free 
by showing that it 
has  finite index in  $\SL(2, R)$.
 Further 
information about the structure of 
$G$ is obtained;
for example, we compute the minimal
subgroup of finite index in $\SL(2,R)$ 
 containing $G$.
\end{abstract}

\maketitle

\section{Introduction}
\label{Intro}
For $x\in \C$, define
\[
A(x) = 
{\small
\begin{bmatrix}
1&\ x\\
0&\ 1
\end{bmatrix}}, \qquad 
 B(x)
= {\small
\begin{bmatrix}
1&\ 0\\
x&\ 1
\end{bmatrix}}.
\]
Let $G(x)$ be the 
subgroup of $\SL(2,\C)$ generated 
by $A(x)$ and $B(x)$, commonly called 
a  (\emph{parabolic}) 
\emph{M\"{o}bius group}.
 Testing freeness of 
M\"{o}bius groups is a well-studied 
problem; see, e.g., 
\cite{Beardon,Brenner,Farbman,Gutan,
Koberda,LyndonUllman}.
Sanov~\cite{Sanov} proved that $G(2)$ 
is free, while Brenner~\cite{BrennerClassic} 
proved that $G(x)$ is free (of rank $2$)
for all $x$ such that $|x|\geq 2$.
Hence, if $x$ is algebraic and
$|\bar{x}|\geq 2$ for some
algebraic conjugate $\bar{x}$ 
of $x$, then $G(x)$ is 
free~\cite[p.~1388]{LyndonUllman}. 
Also, $G(x)$ is free if $x$ is
transcendental~\cite[pp.~30--31]{Wehrf}.

It is unknown whether $G(x)$ is free for 
any rational $x\in(0 , 2)$.
Overall, testing freeness of matrix 
groups is difficult. The 
problem may be undecidable; 
note that testing freeness of 
matrix semigroups is 
undecidable~\cite{Semifree}.
However, we can effectively
decide virtual solvability~\cite{Tits}, 
and so, by the Tits alternative, can
decide in practice 
 whether a given  finitely 
generated linear group contains a
non-abelian free subgroup (without 
producing one).

On the other hand, non-freeness of $G(x)$ 
has been justified for infinitely many 
rationals in $(0 , 2)$, and
there is a set of irrational 
algebraic $x$ that is dense in 
$(-2,2)$ and for which $G(x)$ 
is non-free~\cite[p.~528]{Beardon}.

We take a different approach to certifying
non-freeness of M\"{o}bius groups
in $\SL(2,\Q)$, that
 applies recently developed methods 
to compute with Zariski dense 
matrix groups~\cite{Density,SArithm}. 
We now introduce some basic terms. 
For an integer $b>1$, 
let $R$ be the localization 
$\Z[\frac{1}{b}]= \allowbreak
\{ r/b^i  \, |\, \allowbreak 
 r\in \Z, \, i\geq 0\}$ of $\Z$.
Denote  $\SL(2,R)$ by $\Gamma$.
If $I\subset R$ is a non-zero 
proper ideal, then  there is a unique 
integer $t>1$ coprime to $b$ such
that $I = tR$.
Let $\varphi_{I} \colon 
\Gamma \rightarrow \SL(2, R/I)$ be the 
associated congruence homomorphism. 
The kernel $\Gamma_{I}$ of $\varphi_I$
is a \emph{principal congruence subgroup} 
(PCS) of  $\Gamma$. We say that 
the \emph{level} of $\Gamma_I$ is $t$, 
and write $\varphi_t$, $\Gamma_t$
for $\varphi_I$, $\Gamma_I$, 
respectively. The following notation 
will also be used: $1_r$ is 
the $r\times r$ identity 
matrix, $\Z_k:= \Z/k\Z$, and 
$\F_p$ is the field of size $p$.

Let $S$ be the set of reciprocals of 
the primes dividing $b$. A finite index
subgroup of $\Gamma$ is said to be 
\emph{$S$-arithmetic}.
By \cite{Mennicke,Serre1},  
$\Gamma$ has the \emph{congruence 
subgroup property} (CSP):
each $S$-arithmetic 
subgroup $H$ of $\Gamma$
 contains a PCS.
The \emph{level} of $H$ is defined 
to be the level of the 
maximal PCS in $H$ (the PCS 
 in $H$ with smallest 
possible level). 

Throughout, unless stated otherwise,
 $m = \frac{a}{b}\in \Q$ where 
$a$ is a positive integer coprime to $b$.
(Note that $G(\frac{1}{b}) = \Gamma$.)
Since $G(m)\leq \Gamma$ is Zariski dense 
in $\SL(2)$,
the intersection of all $S$-arithmetic 
subgroups of $\Gamma$ that contain $G(m)$
is $S$-arithmetic~\cite{Density,SArithm}. 
This intersection,
called the \emph{arithmetic closure} 
of $G(m)$, is denoted $\mathrm{cl}(G(m))$.
We define the level of $G(m)$  
to be the level of $\mathrm{cl}(G(m))$.

If $G(m)$ is $S$-arithmetic then it is 
not free (see Section~\ref{NFConditions}).
Since  $\Gamma$ is finitely presented, 
one can attempt to prove $S$-arithmeticity 
of $G(m)$ by coset enumeration.
This necessitates
determining a presentation of $\Gamma$.
  Algorithms for that task are 
developed in Section~\ref{Pres}.
Then in Section~\ref{tests} we report 
on experiments carried out 
 using our {\sf GAP}~\cite{Gap} 
implementation of the algorithms.
There we demonstrate $S$-arithmeticity
 (hence non-freeness) of $G(m)$ for a range 
of rational $m \in (0,2)$.
Although non-freeness of some such $G(m)$ 
was already known, our experiments
illustrate the connection between 
arithmeticity and non-freeness of $G(m)$. 

Moreover, we provide essential information 
about the structure and properties 
 of $G(m)$, 
covering also the case that $G(m)$ is a 
\emph{thin matrix group}~\cite{Sarnak}, 
i.e., of infinite index in 
$\Gamma$. Specifically, we prove that
$G(m)$ has level $a^2$
(Theorem~\ref{Conjecture}). Hence,
if $G(m)$ is $S$-arithmetic, then we 
can name the maximal PCS in $G(m)$, 
and so readily test membership of 
elements of $\Gamma$ in $G(m)$. 
The membership testing
 problem continues to attract attention;
see, e.g., \cite{Shpilrain,Esbelin}. 

We are not aware of any rational 
 $m\in (0,2)$ such that  
$G(m)$ is thin. If there are none,
then this would explain the 
lack of free $G(m)$ for these $m$.
Conversely, can a thin 
$G(m)$ be non-free? 
Within $\SL(2, \Z)$, the situation is
more settled: $G(2)$ is a free subgroup
of finite index, whereas  if $m\geq 3$ 
then $G(m)\leq \SL(2,\Z)$ is free and
thin~\cite[Theorem~3]{Shpilrain}
(indeed, for integers $m>5$, the normal 
closure of $G(m)$ in $\SL(2,\Z)$ is 
thin~\cite[p.~31]{MennickeUni}).
Famously, $\SL(2,\Z)$ does \emph{not} 
have the CSP.
We remark that $|\Gamma:G(m)|=\infty$
for any integer $m\geq 1$.

This paper is based on preliminary 
work presented at 
the ICERM meeting `Computational Aspects 
of Discrete Subgroups of Lie Groups'
(June 14--18, 2021).

\section{Non-freeness criteria
for M\"{o}bius groups}
\label{NFConditions}

Set $A=A(m)$, $B=B(m)$, and $G=G(m)$.
Each element of $G$ is a word 
 $W_n = A^{\alpha_1}B^{\beta_1} 
 \cdots A^{\alpha_n}B^{\beta_n}$ 
where $n \geq 1$ and the 
 $\alpha_i, \beta_j$ are 
integers, all of which are non-zero except
possibly $\alpha_1$, $\beta_n$.
If $G$ is not freely generated by 
$A$ and $B$, then $G$ is not 
free~\cite[p.~1394]{LyndonUllman}.
Moreover, $G$ is free and freely 
generated by $A, B$ if and only if 
 $W_n\neq 1_2$ for each $W_n$ with 
all exponents $\alpha_i, \beta_j$ 
non-zero~\cite[p.~158]{Newman}.

 Non-freeness testing of $G$ 
by a number of 
authors~\cite{Beardon,Gutan,LyndonUllman}
depends on finding words 
of special form in $G$. 
The following is a criterion of such type.
\begin{lemma}\label{NormBasic}
If $N_G(\langle A\rangle)$ is
non-cyclic then $G$ is not free.
\end{lemma}
\begin{proof}
The normalizer is upper triangular, hence
solvable.
\end{proof}

Since $G\leq G(\frac{m}{n})$,
 non-freeness (respectively, $S$-arithmeticity) 
of $G$ implies that of $G(m/n)$.
So in searching for 
 $m\in (0,2)$
such that $G$ is non-free, we can 
restrict to $m\in \allowbreak (1,2)$ if we wish. 

Detecting non-identity elements 
of finite order is another way to 
prove non-freeness~\cite{Charnow,Farbman}.
In \cite{Charnow}, it is shown that 
$G(r)$ for $r\in \Q$ has a non-identity 
element of finite order if and only if 
$\frac{1}{r}\in \Z$. One direction is 
simple: 
$\SL(2, \Z)\subseteq G(\frac{1}{b})$,
so, e.g., $-1_2\in G(\frac{1}{b})$.
An elementary proof of the converse is 
given below.
\begin{proposition}
For $m= \frac{a}{b}$ where $a>1$, 
$G=G(m)$ is torsion-free.
\end{proposition}
\begin{proof}
If $h\in G$ has prime order, then 
$h$ is conjugate to an element of
$\SL(2,\Z)$. The finite order elements
of $\SL(2,\Z)$ are known by a result of
Minkowski~\cite[p.~179]{Newman}; of these,
only $-1_2$ is possibly in 
$G\leq \Gamma_a$.
Hence $m = \frac{2}{b}$, $b$ odd.
But this cannot be. For an element of $G$
looks like
\[
\begin{bmatrix}
1+m^2f_1(m)& mf_2(m)\\
mf_3(m) & 1+m^2f_4(m)\\
\end{bmatrix}
\]
where $f_i(m)\in \Z[m]$, $1\leq i\leq 4$
(\cite[p.~747]{Charnow}),
and after clearing out denominators 
in the equation $2f_1(m) = -b^2$, we get a 
contradiction against $b$ odd.
\end{proof}

Our preferred criterion for non-freeness 
testing follows.
\begin{proposition}\label{Fundamental}
Suppose that $G$ is $S$-arithmetic.
Then $G$ is not free.
\end{proposition}
\begin{proof}
Since $\Gamma$ has the CSP,
  $\Gamma_c\leq G$ for some $c\geq 2$.
 In turn, $\Gamma_c$ contains  
$\langle A(cR)\rangle 
= \{ A(cx) \mid x\in R\}$. The latter is 
 isomorphic to the additive group 
$(cR)^+$ of the ideal $cR\subseteq R$,
and  $(cR)^+$  is non-cyclic. 
\end{proof}

So, our non-freeness test for 
M\"{o}bius subgroups  
 is really a spinoff
from attempts to prove $S$-arithmeticity
 in $\Gamma$.

\section{Exploiting Zariski density 
of M\"{o}bius groups}
\label{thin}

In this section we establish 
properties of $G=G(m)$ that 
derive from Zariski density of 
$G$ and the fact that 
$\Gamma$ has the CSP. 
In particular, we 
find a generating
set for $\mathrm{cl}(G)$.

Denote the set of 
prime divisors of $k\in \Z$ by $\pi(k)$.
Let $\Pi(H)$ be the (finite)
 set of all primes $p$ modulo which dense
$H\leq \Gamma$ does not surject onto 
$\SL(2, p)$, i.e.,
\[
\Pi(H)=\{ p \in \Z \mid p \, \text{ prime}, 
\, \mathrm{gcd}(p,b)=1,
\, \varphi_p(H) \neq \SL(2,p)\}.
\] 
\begin{lemma}\label{PrimesForDense}
$\Pi(G)=\pi(a)$.
\end{lemma}
\begin{proof}
If $p \, | \, a$ then $G\leq \Gamma_p$; 
so  $\pi(a)\subseteq \Pi(G)$.
If $p\nmid a$ then
$\varphi_p(G) = \allowbreak \SL(2, p)$,
because $\SL(2,p)=\langle
A(k), B(k)\rangle$ for any
$k\in \F_p\setminus \{0\}$.
\end{proof}

Let $l$ be the level of $G$.
Results from \cite{SArithm} imply 
that $\Pi(G) \setminus \{ 2,3,5\}= 
\pi(l)\setminus\{2,3,5\}$.
We will see  that
$l=a^2$ (Theorem~\ref{Conjecture}); 
so $\Pi(G)=\pi(a)=\pi(l)$ without
exception.

\begin{remark} 
\label{SurjectionEquivalence}
If $G$ surjects onto 
$\SL(2, p)$ modulo the prime 
 $p$, then $G$ surjects
onto $\SL(2, \Z_{p^k})$ modulo
$p^k$ for all $k\geq 1$.
\end{remark}

\begin{lemma}\label{AHlem}
Let $a=p^e c$ where $p$ is
a prime, $e\ge 1$, and 
$\gcd(p,bc)=1$.
Then 
\[
\Gamma_{p^{2e}}\le 
 G \hspace{1pt}\Gamma_{p^f}
\]
for any $f\ge 0$, but
\[
\Gamma_{p^{2e-1}} \not \le  
 G\hspace{1pt} \Gamma_{p^f}
\]
if $f\ge 2e+1$.
\end{lemma}
\begin{proof}
The lemma is trivially true if
$f\le 2e$.  
Suppose that $f> 2e$. 
Then $\Gamma_{p^f}\le\Gamma_{p^{2e+1}}
\le \Gamma_{p^{2e}}$.

Since $\gcd(p,b c)=1$, there 
is a positive integer $i$ such that 
$i\cdot\frac{c}{b}\equiv 1\bmod{p^f}$, 
so
\[
A(m)^i \equiv A(p^e) 
\quad 
\text{ and } \quad
B(m)^i \equiv B(p^e) \quad \bmod{p^f}.
\]
Hence $\varphi_{p^f}(G)
=\varphi_{p^f}(G(p^e))$.
We therefore prove the two 
assertions with $G(p^e)$ in 
place of $G$.

Setting $A(p^e)=x$ and $B(p^e)=y$,
we have
\[
[x,y]=
\begin{bmatrix}
1+p^{2e}+p^{4e}&p^{3e}\\
-p^{3e}&1-p^{2e}
\end{bmatrix}
\equiv
\begin{bmatrix}
1+p^{2e}&0\\
0&1-p^{2e}
\end{bmatrix} \quad 
\bmod{p^{2e+1}}.
\]
It follows that 
$U:=\varphi_{p^{2e}}(G(p^e))$ 
is abelian, consisting of the 
$p^{2e}$  distinct images
\[
\varphi_{p^{2e}}
\Big(
\begin{bmatrix}
1&sp^e\\
tp^{e}&1
\end{bmatrix}\Big)
\]
for $0\le s,t< p^e$.
But then $U$ does not contain
\[
\varphi_{p^{2e}}
\Big(
\begin{bmatrix}
1+p^{2e-1}&p^{2e-1}\\
-p^{2e-1}&1-p^{2e-1}
\end{bmatrix}
\Big).
\]
Thus
$\varphi_{p^{2e}}(\Gamma_{p^{2e-1}})
\not\le U$.

It remains to prove the first 
assertion, and this we do by
induction on $f$. For the base step
$f = 2e+1$, note that the images of 
$[x,y]$, $x^{p^e}$, $y^{p^e}$ under 
$\varphi_{p^{2e+1}}$ are 
\[
\begin{bmatrix}
1+p^{2e}&0\\
0&1-p^{2e}
\end{bmatrix} , \quad
\begin{bmatrix}
1&\ p^{2e}\\
0&\ 1
\end{bmatrix} ,  
\quad 
\begin{bmatrix}
1&\ 0\\
p^{2e}&\ 1
\end{bmatrix} .
\]
These generate the elementary 
abelian group
$\varphi_{p^{2e+1}}(\Gamma_{p^{2e}})$, 
whose elements are of the form
$1_2+p^{2e}u$
where $u$ has entries in 
$\{0,\ldots,p-1\}$
and $\trace(u)\equiv
\allowbreak  0\bmod p$.

Assume now that the statement is true 
for $f=k\geq 2e+1$.
Then $\varphi_{p^k}(\Gamma_{p^{k-1}})
\leq \varphi_{p^k}(G(p^e))$. 
Thus, for each
$1_2+p^{k-1} u\in \varphi_{p^k}(\Gamma_{p^{k-1}})$
where $u$ is a $\{0,\ldots,p-1\}$-matrix 
with zero trace modulo $p$,
there exist $v\in G(p^e)$ and some $w$
such that $v=1_2+p^{k-1} u+p^k w$.
Then $v^p\equiv 1_2+p^k u \allowbreak
\bmod{p^{k+1}}$,
which implies that
$\varphi_{p^{k+1}}(\Gamma_{p^k})\le 
\varphi_{p^{k+1}}(G(p^e))$. Hence 
$\Gamma_{p^{2e}} \leq 
\langle G(p^e), \Gamma_{p^{k+1}}\rangle$
by the inductive hypothesis.
\end{proof}
\begin{remark}
Lemma~\ref{AHlem} is valid for $b=1$.
\end{remark}

\begin{lemma}\label{OrderFormula}
In the notation of 
Lemma{\em ~\ref{AHlem}},
$\varphi_{p^{2e}}(G)\cong
 C_{p^e}\times C_{p^e}$.
Hence
\[
|\SL(2,\Z_{p^{2e}}):
 \varphi_{p^{2e}}(G)|=
p^{4e}- p^{4e-2}=
p^e \hspace{1pt}|\SL(2,\Z_{p^e})|.
\]
\end{lemma}
\begin{proof}
We observed that 
$\varphi_{p^{2e}}(G)=
 \varphi_{p^{2e}}(G(p^e))\cong
 C_{p^e}\times C_{p^e}$ in the 
proof of Lemma~\ref{AHlem}.
Also, $|\SL(n,\Z_{p^k})| = 
p^{(n^2-1)(k-1)} 
 \hspace{.5pt} |\SL(n,p)|$.
\end{proof}

We gather together various observations about the 
structure of $\SL(2,\Z_{p^e})$, $p$ prime,
that are
needed in the subsequent proof.
\begin{lemma} \label{MiscSL2}\
\begin{itemize}
\item[{\rm (i)}] Each proper normal subgroup of 
$\SL(2,\Z_{p^e})$ has index divisible by $p$.
\item[{\rm (ii)}] 
Let $l>1$ be an integer with prime 
factorization $l=p_1^{e_1}\cdots p_k^{e_k}$. 
Let $K$ be a subgroup of $\SL(2,\Z_l)$ 
such that $\varphi_{p_i^{e_i}}(K)= \SL(2,\Z_{p^{e_i}})$
for all $i$, $1\leq i\leq k$. Then $K=\SL(2,\Z_l)$.
\end{itemize}
\end{lemma}
\begin{proof}
(i) \hspace{1pt} 
A counterexample would arise from a 
proper normal subgroup of the quotient $\PSL(2,p)$.
Hence the statement is clear for $p\geq 5$, and 
it follows for $p\in \{2, 3\}$ by inspection.

(ii) \hspace{1pt} We proceed by induction on $k$, the
base step being trivial. 
Write $l=rp^e$ where $p$ is the largest prime divisor 
of $l$ and $\mathrm{gcd}(p,r)=1$. Then 
$\SL(2,\Z_l) \cong \SL(2,\Z_{p^e})\times \SL(2,\Z_r)$;
so $K$ is a subdirect product of 
$\varphi_{p^e}(K)= \SL(2,\Z_{p^e})$ and $\varphi_r(K)$.
The inductive hypothesis gives $\varphi_r(K)= \SL(2,\Z_r)$.
Then (i) forces
$G$ to be the full direct product 
$\SL(2,\Z_{p^e})\times \SL(2,\Z_r)$. This completes the
 proof by induction.
\end{proof}

The next result
 is the main one of this section, and
it facilitates our experiments in the final section. 
\begin{theorem}\label{Conjecture}\
\begin{itemize}
\item[{\rm (i)}]
$\cl(G)= G\hspace{1pt} \Gamma_{a^2}$
has level $a^2$. 
\item[{\rm (ii)}]
$|\Gamma:\cl(G)|=a\cdot|\SL(2,\Z_a)|$.
\end{itemize}
\end{theorem}
\begin{proof}
(i) \hspace{1pt} Let $l$ be the level of $G$.
If $p^e>1$ is 
the largest power of the prime
$p$ dividing $a$,  then 
$G \hspace{1pt}\Gamma_{p^{2e}}$ 
has level $p^{2e}$ by Lemma~\ref{AHlem}. 
Since $H:=\cl(G)= G\hspace{1pt}\Gamma_l \leq 
G \hspace{1pt}\Gamma_{p^{2e}}$ 
and therefore $\Gamma_l\leq \Gamma_{p^{2e}}$, 
we see that $l$ is divisible by $p^{2e}$. 
Thus $a^2$ divides $l$. 

Next we prove $\pi(l)\subseteq \pi(a)$, 
i.e., $\pi(l) = \pi(a)$. 
To this end, suppose that 
$l=rk$ where $r>1$, $\pi(k) = \pi(a)$,
 and $\mathrm{gcd}(r,a)=1$.
By Lemma~\ref{PrimesForDense}, 
Remark~\ref{SurjectionEquivalence}, 
 and Lemma~\ref{MiscSL2},
 $\varphi_{r}(G)= \SL(2,\Z_{r})$.
Since 
$\SL(2,\Z_l)\cong \SL(2,\Z_{r})\times 
\SL(2,\Z_k)$,  
it follows that $\varphi_l(H)= \varphi_l(G)$ 
is a subdirect product 
of  $\SL(2,\Z_{r})$ and 
 $\varphi_k(H)$.
Each proper quotient of $\SL(2,\Z_{r})$
has order divisible by a prime in $\pi(r)$, whereas
 $\varphi_k(H)$ has order divisible only by the 
primes in $\pi(a)$. Thus $\varphi_l(H) \cong 
\SL(2,\Z_{r}) \times \varphi_k(H)$, 
and as a consequence $\Gamma_k \leq H$. But
$\Gamma_l$ is the PCS of least level in $H$.

We have now proved that $\pi(l)= \pi(a)$.
Therefore $\varphi_l(H)$ is a direct 
product of $p$-groups $\varphi_{p^f}(H)$ 
for $p$ ranging over
$\pi(a)$. Suppose that $l>a^2$; 
say $p^f$ divides $l$
where $f>2e$ and $p^e$ is the largest power of $p$
dividing $a$. 
We infer from $\Gamma_{p^{2e}} \leq 
G\hspace{1pt}\Gamma_{p^f}$ 
 that $\varphi_{l}(\Gamma_{l/p^{f-2e}}) \leq
\varphi_{l}(H)$, so $\Gamma_{l/p^{f-2e}} \leq H$:
 contradiction. Hence $l = a^2$.
 
(ii) \hspace{1pt} This follows from
Lemma~\ref{OrderFormula}.
\end{proof}

\begin{corollary}\label{CongQuot}
$\cl(G)/\Gamma_{a^2}\cong C_a\times
C_a$.
\end{corollary}

The proof of Theorem~\ref{Conjecture}
is independent of the denominator $b$,  
so we have proved additionally that 
$G(a)\leq \SL(2,\Z)$  has level 
$a^2$. 
Cf.~\cite[Proposition~1.12]{Arithm}:
for $n\geq 3$, the `elementary group'
in $\SL(n,\Z)$ generated by all 
matrices $1_n+me_{ij}$ with $i\neq j$
($e_{ij}$ has $1$ in position $(i,j)$ and
zeros elsewhere) contains the PCS of 
level $m^2$.  

In line with \cite{SuryVenky}
(and cf.~\cite[Proposition~1.10]{Arithm}),
one can show that $\Gamma_{a^2}$
is generated by 
\[
 A(am), \hspace{3pt}
 B(am), \hspace{3pt}
 B(am)^x, \hspace{4pt}
\mbox{where} \hspace{4pt}
x = {\scriptsize
\begin{bmatrix}
-1&1 \ \\
\phantom{-}0& 1 \
\end{bmatrix}}.
\]
Also, at least when $b$ is prime,
$\Gamma_{a^2}$ is the normal closure  
$\langle A(am)\rangle^{\Gamma}=
\langle B(am)\rangle^\Gamma$ of 
$G(am)$ in $\Gamma$ 
(this implies 
the  CSP; see
\cite{Mennicke}).
Thus, we get an explicit 
generating set of $\cl(G)$. 
 
Theorem~\ref{Conjecture} and 
Corollary~\ref{CongQuot} afford 
 further insights. Lifting up 
to $G$ from a presentation of 
$\varphi_{a^2}(G)\cong C_a \times C_a$ 
by the `normal generators' technique 
(see, e.g.,
\cite[$\S \hspace{1pt} 3.2$]{Tits})
yields that $G\cap \Gamma_{a^2}$ 
is the $G$-normal closure of
\[
\langle \hspace{.5pt} A(am),
\hspace{.5pt} B(am), 
\hspace{.5pt}[A(m),B(m)]
\hspace{.5pt} \rangle .
\]

In summary, using the notation 
$\leq_{\mathrm f}$, $<_{\infty}$
to denote a subgroup of finite or 
infinite index, respectively:
\begin{enumerate}
\item if $G$ is thin then
 $G(am) \leq G 
<_{\infty} \mathrm{cl}(G) 
\leq_{\mathrm f}\Gamma_a$;
\item if $G$ is $S$-arithmetic then
$G(am) \leq \Gamma_{a^2} 
\leq_{\mathrm f} G 
\leq_{\mathrm f} \Gamma_a$;
\item $G$ is $S$-arithmetic 
if and only if 
$\langle A(am)\rangle^{\Gamma} 
\leq  G$.
\end{enumerate} 

\vspace{2pt}

When $G$ is $S$-arithmetic,
an extra benefit of Theorem~\ref{Conjecture}
is that it enables us to test membership of 
$g \in \Gamma$ in $G$.
The idea is obvious:
$g\in G$ if and only if 
$\varphi_{a^2}(g)\in \varphi_{a^2}(G)$.
More generally, we could get a negative 
answer to the question of whether $g$ is in 
$G$ by testing membership 
 of $g$ in $\mathrm{cl}(G)$. 

\section{Experimentation}
\label{Exp}

\subsection{Computing a presentation 
of $\SL(2, R)$}\label{Pres}

A key requirement for the  
experiments, which feature 
 Todd--Coxeter coset enumeration,
 is a presentation of $\Gamma$.
Our algorithm to compute one
 is based on the next theorem, 
due to Ihara~\cite[Corollary~2, p.~80]{Serre}. 
\begin{theorem}\label{Ihara}
For any prime $p$, 
\[
\SL(2, \Z[\textstyle 1/p])\cong 
\SL(2, \Z) *_{\Gamma_0(p)}\SL(2, \Z),
\]
the free product with amalgamation 
$\Gamma_0(p)$, where $\Gamma_0(p)$
 comprises all
matrices in $\SL(2, \Z)$ 
that are upper triangular modulo $p$.
\end{theorem}

Theorem~\ref{Ihara} can be expanded 
iteratively to $\SL(2, \Z[\textstyle 1/b])$
for composite $b$~\cite[p.~80]{Serre}.
If $k$ is an integer not divisible by
 $p$, then  $\SL(2, \Z[1/pk])$
is the amalgamated free product of two 
copies of $\SL(2, \Z[1/k])$,
 with the amalgamated subgroup as before 
comprising all mod-$p$ upper triangular matrices.

We apply Theorem~\ref{Ihara}
in standard fashion (see
\cite[pp.~80--81]{Serre}), 
taking the presentations 
\[
\langle \hspace{.5pt} s, t \mid 
 s^4 = 1, \hspace{.5pt} (st)^3=s^2 
\hspace{.5pt} \rangle \quad
\mbox{and} \quad 
\langle \hspace{.5pt} x_p, y_p \mid 
 x_p^4 = 1, \hspace{.5pt} 
(x_py_p)^3=x_p^2 \hspace{.5pt} \rangle
\] 
of $\SL(2,\Z)$, where
\[
s=
{\scriptsize \begin{bmatrix}
\phantom{-}0&1\\
-1&0\\
\end{bmatrix}}, \quad
t= 
{\scriptsize \begin{bmatrix}
1&0\\
1&1\\
\end{bmatrix}}, \quad
x_p =
{\scriptsize \begin{bmatrix}
\phantom{-}0&\frac{1}{p}\\
& \vspace{-6.75pt}\\
-p&0\\
\end{bmatrix}}, \quad 
y_p =
{\scriptsize \begin{bmatrix}
1&\! -\frac{1}{p}\\
& \vspace{-6.75pt}\\
0&\! \phantom{-}1\\
\end{bmatrix}}.
\]
We construct
a generating set for 
$\Gamma_0(p)$ whose elements are
 Schreier generators for the 
stabilizer of the subspace 
$\langle (1,0)^\top\rangle
\subseteq \F_p^2$ under the 
natural action of $\SL(2,\Z)$ 
modulo $p$. 
It is easily seen that
$|\SL(2, \Z) : \Gamma_0(p)| = p+1$.
Suppose that $\Gamma_0(p)$ is given by
a generating set of matrices $w_i = w_i(x,y)$
that are words in $x=x_p$ and $y=y_p$.  
Utilizing the {\sf GAP} package 
ModularGroup~\cite{GAPModularGroup},
we express each $w_i$ as a word 
$w_i'(s,t)$ in $s$ and $t$.
Then $\SL(2, \Z[1/p])$  
has presentation 
\begin{align*}
\langle \hspace{.5pt} s,t,x,y \ | &  \
s^4=1 , \hspace{.5pt} (st)^3=s^2, 
\hspace{.5pt} 
x^4=1, \hspace{.5pt} (xy)^3=x^2,\\
 &  \ 
w_1(x,y)=w_1'(s,t), 
\ldots , w_{k}(x,y)=w_{k}'(s,t) \rangle 
\end{align*}
for some $k\leq p+2$;
of course, this may simplify. 
However, eliminating redundancy
is usually not worthwhile, because 
it comes at the cost of longer words. 

Below are a few examples that illustrate the 
method.

\vspace{5pt}

\noindent \textbf{$\bullet$} \,
 Let $p=5$. Then  $\Gamma_0(p)$ 
is generated by
\[
x^{-2}, xyx^{-1}, y^{5}, y^{2}xy^{-2}, yx^{-1}y .
\] 
Actually, the fifth listed generator is 
 redundant. Now
\[
x^{-2} =
{\scriptsize \begin{bmatrix}
-1&\! \phantom{-}0\\ 
\phantom{-}0&\! -1\\ 
\end{bmatrix}}, \quad
xyx^{-1} =
{\scriptsize \begin{bmatrix}
1&0\\ 
5&1\\ 
\end{bmatrix}}, \quad
y^{5} = {\scriptsize \begin{bmatrix}
1&\! -1\\ 
0&\! \phantom{-}1\\ 
\end{bmatrix}},\quad
y^{2}xy^{-2} =
{\scriptsize \begin{bmatrix}
\phantom{-}2&\! \phantom{-}1\\ 
-5&\! -2\\ 
\end{bmatrix}}.
\]
The corresponding words 
$w'$ are $s^{-2},t^5, (tst)^{-1},
t^{-2}st^{2}$. Hence
\[
\langle \hspace{.5pt} s,t,x,y\mid 
x^{4}, \hspace{.5pt}
xyxyxyx^{-2}, \hspace{.5pt}
s^{4}, \hspace{.5pt}
stststs^{-2}, \hspace{.5pt}
x^{-2}s^{2}, \hspace{.5pt}
xyx^{-1}t^{-5}, \hspace{.5pt}
y^{5}tst, \hspace{.5pt}
y^{2}xy^{-2}t^{-2}s^{-1}t^{2}
\hspace{.5pt} \rangle
\]
is a presentation 
 of $\SL(2,\Z[1/5])$.

\vspace{5pt}

\noindent \textbf{$\bullet$} 
\, We similarly calculate that 
 $\SL(2, \Z[1/7])$ 
has presentation
\begin{align*}
\langle \hspace{.5pt} s,t,x,y\ | & \ 
s^{4}, \hspace{.5pt}
stststs^{-2}, \hspace{.5pt}
x^{4}, \hspace{.5pt}
xyxyxyx^{-2}, \hspace{.5pt}
x^{-2}s^{2}, xyx^{-1}t^{-7},\\
 &  \ 
y^{3}x^{-1}y^{-2}t^{-3}st^{2}, \hspace{.5pt}
y^{-3}x^{-1}y^{2}t^{3}st^{-2}
\hspace{.5pt} \rangle.
\end{align*}
The generating set 
in this example is redundant.

\vspace{5pt}

\noindent \textbf{$\bullet$} \,
 The results for $p=5$ and $p=7$
combine to give 
\begin{align*}
\SL(2,\Z[{\textstyle \frac{1}{35}}])
=
\langle \hspace{.5pt} s,t,x,y,c,d  \ \, | & 
\ s^{4}, \hspace{.5pt} 
stststs^{-2}, \hspace{.5pt}
x^{4}, xyxyxyx^{-2}, \hspace{.5pt}
x^{-2}s^{2}, \hspace{.5pt}
xyx^{-1}t^{-5},\\
 & \ y^{5}tst, \hspace{.5pt}
y^{2}xy^{-2}t^{-2}s^{-1}t^{2}, \hspace{.5pt}
c^{4}, \hspace{.5pt}
cdcdcdc^{-2}, c^{-2}s^{2},\\
 &  \
cdc^{-1}t^{-7}, \hspace{.5pt}
d^{-3}c^{-1}d^{2}t^{3}st^{-2}, \hspace{.5pt}
d^{3}c^{-1}d^{-2}t^{-3}st^{2}
 \hspace{.5pt} \rangle
\end{align*}
where $x=x_5$, $y=y_5$, $c= x_7$, 
and $d= y_7$.

\vspace{5pt}

We implemented the above approach 
as a {\sf GAP} procedure. 
 This accepts an integer $b>1$, and
returns a presentation of $\Gamma = G(1/b)$ 
on $2+2|\pi(b)|$ generators.
In our experience, 
this procedure completed rapidly 
for prime $b\approx \allowbreak 2000$. 
 For composite $b$,
the cost is basically just that of the 
prime factors, since the amalgamation 
is straightforward.
For larger $a$, we find short word 
expressions by systematically 
enumerating all words by 
increasing length.

\subsection{$S$-arithmeticity 
of $G(m)$}\label{tests}

In this section we justify 
 $S$-arithmeticity 
(and thereby non-freeness)
of $G(m)$ for a range of rational 
$m$ between $0$ and $2$. 
 
First, we express the generators
 $A=A(m)$ and $B=B(m)$ of $G=G(m)$ as 
 words in the generators of $\Gamma$. 
 When $m = a/p$ for a prime $p$,
 this is easy:
$A(a/p) = A(1/p)^a = y^{-a}$ 
and $B(a/p)=sy^as^{-1}$.
For larger $a$,
we try to find short word 
expressions by random search. 
Then we take the subgroup of 
$\Gamma$ defined by the
words for $A$ and $B$ 
and carry out Todd--Coxeter coset 
enumeration (within the {\sf GAP} 
package ACE~\cite{actetc}).
If the enumeration terminates, then
rewriting using the modified Todd--Coxeter
algorithm of \cite{Neubueser81} finds 
relators in $A$ and $B$.
Unless $a$ is small, these relators 
tend to be long; e.g., the shortest relator 
discovered for $m=4/11$ is
$A^{121}BA^{-11}B^{2}A^{-121}B^{-1}A^{11}B^{-2}$.
Relator length typically increases 
with the index of $G$ in $\Gamma$ 
(governed by $a$; Theorem~\ref{Conjecture}~(ii)). 
This suggests that searching for short relators 
with particular 
patterns (cf.~\cite{Beardon})
has limited chance of 
proving non-freeness.
Also recall that if $G$ is $S$-arithmetic, then
it contains upper unitriangular elements $u$ 
and $v$ such that $\langle u, v\rangle$ is 
non-cyclic. Once we express
$uvu^{-1}v^{-1}$ as a word
in $A$ and $B$, we have 
 a non-trivial relator in $G$.

Stages in the coset enumeration 
may entail significantly 
(and unboundedly) more cosets than $|\Gamma :G|$, 
so that coset enumeration for $G$ in $\Gamma$ 
may fail to terminate within reasonable time or
available memory, even for small indices. 
We emphasize that failure to terminate is 
\emph{not} a proof that $G$ is thin.

If finite, the index $a\hspace{1pt} |\SL(2,\Z_a)|$ 
of $G(a/b)$ in $\Gamma$ is greater than
$10^6$, $10^7$, $10^8$, 
for $a=\allowbreak 33$, 
$a=\allowbreak 59$, $a=101$, respectively.
With four generators, and ignoring 
overhead, each row in the coset table 
requires at least $4\cdot 2\cdot 8=64$ 
bytes of memory.  This means that 1GB
of memory can store no more than $10^7$ cosets. 
By this reckoning, we expect that coset enumeration 
becomes difficult when $a$ is between $40$ 
and $50$. If the index is finite, then we 
expect that enumeration succeeds if $a\le 25$.

Returning to an earlier point: 
$S$-arithmeticity of $G(a/b)$ implies 
$S$-arithmeticity of $G(a/kb)$
for every integer $k>0$, 
so we could restrict 
experimentation to 
groups $G( a/p)$ where $p$ is 
prime. However, 
the practicality of an attempt 
to prove $S$-arithmeticity of $G(a/b)$ is 
affected by the size of $a$.
Thus, we investigated $G(m)$ 
with $m=a/b < 2$ where $b=p^k$, 
$p\le 23$ prime. 
For each fixed denominator $b$, 
we found an integer 
$a_{\mathrm{max}}$ 
 such that $G(a/b)$ is 
$S$-arithmetic whenever 
 $a\le a_{\mathrm{max}}$. 
Table~1 displays
output of the experiments.

\bigskip

\begin{table}[htb]
\label{Table1}
{\small 
\begin{tabular}{|c||c|c|c|c|c|c|c|c|c|c|c|c|c|c|c|c|}\hline
$a_{\mathrm{max}}$ &
 $3$  &
 $7$  &
$13$  &
$23$  &
$37$  &
 $45$  &
 $57$  & 
 $5$  &
$14$  &
 $31$  &
 $41$  & 
 $9$  &
  $28$  &
 $39$ &
 $11$ &
$25$ 
 \\ \hline
$b$ & 
$2$&
$4$& 
$8$& 
$16$& 
$32$& 
$64$& 
$128$& 
$3$& 
$9$&  
$27$&   
$81$&  
$5$&  
$25$&   
$125$&  
$7$&  
$49$
\\ 
\hline 
\end{tabular}
}

\medskip

{\small 
\begin{tabular}{|c||c|c|c|c|c|c|}\hline
$a_{\mathrm{max}}$ &
$12$ &
 $23$  & 
 $15$  & 
 $14$  & 
  $14$  & 
 $14$ \\ \hline
$b$ &   
$11$&  
$121$&  
$13$&  
$17$&  
$19$&   
$23$ 
\\ 
\hline 
\end{tabular}
}
\medskip
\caption{Values of 
 $a_{\mathrm{max}}$
and $b$ such that 
$G(m)$ is $S$-arithmetic for all
$m=\frac{a}{b}$, where
$a\leq a_{\mathrm{max}}$}
\end{table}

\vspace{-5pt}

We also managed to prove that 
$G(m)$ is $S$-arithmetic for the following
$m= a/b$ where $a$ exceeds 
$a_{\mathrm{max}}$ for the given $b$:
$63/64, 65/64$,
$44/125, 51/125, 57/125$,
$29/49$.

Unsuccessful enumerations were 
re-attempted for comparatively 
small $a$. Here we tried
strategies such as increased memory, 
simplifying the
presentation, or use of intermediate 
subgroups.
None of these helped, leading us
to suspect that the 
maximum numerator values 
up to $30$ or so are likely to be
correct.

Our experiments show that many $G(m)$ 
previously known to be non-free 
 are $S$-arithmetic (see
\cite{Beardon,Farbman,LyndonUllman}).
As a single complementary example,
 we proved that $G(11/19)$ 
is not free,
whereas freeness of $G(11/19)$ was 
unresolved in~\cite{Beardon}.

Our {\sf GAP} code is posted at
\url{https://github.com/hulpke/arithmetic}.

\vspace{1.5pt}

\medskip

\noindent \textbf{Acknowledgment.}
 We thank Prof.~Vladimir Shpilrain 
for helpful 
conversations. 
We also thank Mathematisches 
Forschungsinstitut 
Oberwolfach and Centre 
International de Rencontres 
Math\'{e}matiques, Luminy, 
for  
hosting our visits under their
 Research Fellowship and 
Research in Pairs programmes.
The third author's work has 
been supported in part by 
NSF~Grant~DMS-1720146 and 
Simons Foundation Grant~852063, 
which are gratefully acknowledged.

\medskip

\bibliographystyle{amsplain}

\begin{thebibliography}{10}

\bibitem{Beardon}
A.~F.~Beardon, 
\textit{Pell's equation and two generator 
free M\"{o}bius groups},
Bull. London Math. Soc. \textbf{25} 
(1993), no.~6, 527--532.

\bibitem{BrennerClassic}
J.~L.~Brenner,
 \textit{Quelques groupes libres de matrices},
 C. R. Acad. Sci. Paris \textbf{241} (1955),
 1689--1691.

\bibitem{Brenner}
J.~L.~Brenner, R.~A.~Macleod, 
and D.~D.~Olesky,
\textit{Non-free groups generated by 
two {$2 \times 2$} matrices}, 
Canadian J. Math. \textbf{27} (1975), 
237--245.

\bibitem{Semifree}
J.~Cassaigne, T.~Harju, and J.~Karhum\"{a}ki,
\textit{On the undecidability of 
freeness of matrix semigroups},
Internat. J. Algebra Comput.,
  \textbf{9} (1999),
    no.s~3-4, 295--305.

\bibitem{Charnow}
A.~Charnow, \textit{A note on torsion free 
groups generated by pairs of matrices},
Canad. Math. Bull. \textbf{17} (1974/75), 
no.~5, 747--748.

\bibitem{Shpilrain}
A.~Chorna, K.~Geller, and V.~Shpilrain, 
\textit{On two-generator subgroups in 
{$SL_2(\Bbb{Z})$}, {$SL_2(\Bbb{Q})$}, 
and {$SL_2(\Bbb{R})$}}, 
J.~Algebra \textbf{478} (2017), 367--381.

\bibitem{Arithm}
A.~S.~Detinko, D.~L.~Flannery, and 
A.~Hulpke, \textit{Algorithms for
arithmetic groups with the congruence 
subgroup property}, 
J.~Algebra \textbf{421} (2015), 
234--259.	

\bibitem{Density}
A.~S.~Detinko, D.~L.~Flannery, and 
A.~Hulpke, 
\textit{Zariski density and computing 
in arithmetic groups}, 
Math. Comp. \textbf{87} (2018), 
no.~310, 967--986. 

\bibitem{SArithm}
A.~S.~Detinko, D.~L.~Flannery, and 
A.~Hulpke, 
\textit{Zariski density and computing 
with $S$-integral groups}, 
preprint (2022).

\bibitem{Tits}
A.~S.~Detinko, D.~L.~ Flannery, and 
E.~A.~O'Brien,  
\textit{Algorithms for the Tits alternative 
and related problems}, 
J.~Algebra \textbf{344} (2011), 397--406.

\bibitem{Esbelin}
H.-A.~Esbelin and M.~Gutan,  
\textit{Solving the membership problem 
for parabolic {M}\"{o}bius monoids},
Semigroup Forum \textbf{98} (2019), 
no.~3, 556--570. 

\bibitem{Farbman}
S.~P.~Farbman, 
\textit{Non-free two-generator subgroups 
of {${\rm SL}_2(\bold Q)$}}, 
Publ.~Mat. \textbf{39} (1995), no.~2, 379--391.

\bibitem{actetc}
G.~Gamble, A.~Hulpke, G.~Havas, and 
C.~Ramsay,
The {\sf GAP} package ACE
(Advanced Coset Enumerator).
\url{https://www.gap-system.org/Packages/ace.html}

\bibitem{Gap}
The~GAP Group, 
\emph{{GAP} -- {G}roups, {A}lgorithms, 
and {P}rogramming}.
\url{http://www.gapsystem.org}

\bibitem{Gutan}
M.~Gutan, \textit{Diophantine equations 
and the freeness of M\"{o}bius groups},
 Applied Mathematics \textbf{5} (2014), 
1400--1411. 
\url{http://dx.doi.org/10.4236/am.2014.510132}

\bibitem{GAPModularGroup}
L.~L.~Junk and G.~Weitze-Schmith\"{u}sen,
The {\sf GAP} package 
ModularGroup,
\url{https://github.com/AG-Weitze-Schmithusen/ModularGroup}

\bibitem{Koberda}
S.-h.~Kim and T.~Koberda,
\textit{Non-freeness of groups generated 
by two parabolic elements with small 
rational parameters}.
\url{https://arxiv.org/abs/1901.06375v4}

\bibitem{LyndonUllman}
R.~C.~Lyndon and J.~L.~Ullman, 
\textit{Groups generated by two parabolic 
linear fractional transformations},
Canadian J. Math. \textbf{21} (1969), 
1388--1403.

\bibitem{MennickeUni}
J.~Mennicke,
\textit{Finite factor groups of 
the unimodular group}, 
Ann. of Math.~(2) \textbf{81} (1965), no.~1, 
31--37.

\bibitem{Mennicke}
J.~Mennicke,
\textit{On Ihara's modular group}, 
Invent. Math. \textbf{4} (1967), 202--228.

\bibitem{Neubueser81}
J.~Neub\"{u}ser, 
\textit{An elementary introduction to 
coset table methods in
  computational group theory}, 
	Groups---{S}t. {A}ndrews 1981 
	({S}t. {A}ndrews, 1981), 
	London Math. Soc. Lecture Note Ser., vol.~71, 
	Cambridge Univ. Press,
  Cambridge-New York, 1982, pp.~1--45.

\bibitem{Newman}
M.~Newman, \textit{Integral matrices}, 
Academic Press, New York-London, 1972.

\bibitem{Sanov}
I.~N.~Sanov, 
\textit{A property of a representation of 
a free group},
Doklady Akad. Nauk SSSR (N. S.) 
\textbf{57} (1947), 657--659.

\bibitem{Sarnak}
P.~Sarnak, \textit{Notes on thin matrix groups}, 
Thin groups and superstrong approximation, 
343--362, Math. Sci. Res. Inst. Publ. \textbf{61}, 
Cambridge Univ. Press, Cambridge, 2014.

\bibitem{Serre1}
J.-P.~Serre, 
\textit{Le probl\`{e}me des groupes 
de congruence pour $\mathbf{SL}_2$}, 
Ann. of Math.~(2) \textbf{92} (1970), 
489--527.

\bibitem{Serre}
J.-P.~Serre, \textit{Trees}, 
Springer-Verlag, Berlin-New York, 1980.

\bibitem{SuryVenky}
B.~Sury and T.~N.~Venkataramana,
\textit{Generators for all principal 
congruence subgroups of $\SL(n, Z)$ with $n \geq 3$}, 
Proc. Amer. Math. Soc. \textbf{122} (1994), no.~2,
355--358.

\bibitem{Wehrf}
B.~A.~F.~Wehrfritz, \textit{Infinite linear groups}, 
Springer-Verlag, New York-Heidelberg, 1973.

\end{thebibliography}

\end{document}